\documentclass{amsart}
\usepackage{mathabx}

\usepackage{}
\copyrightinfo{2009}{American Mathematical Society}

\newtheorem{theorem}{Theorem}[section]
\newtheorem{lemma}[theorem]{Lemma}
\newtheorem{proposition}[theorem]{Proposition}
\newtheorem{corollary}[theorem]{Corollary}

\usepackage{tikz}
\usepackage{calc}
\usetikzlibrary{decorations.markings}
\usetikzlibrary{arrows}
\usetikzlibrary{arrows.meta}
\tikzstyle{vertex}=[circle, draw, inner sep=0pt, minimum size=1pt]
\newcommand{\vertex}{\node[vertex]}

\theoremstyle{definition}

\theoremstyle{remark}

\numberwithin{equation}{section}

\begin{document}

\title{Divisor graph of complement of $\Gamma(R)$}

\author{Ravindra Kumar}
\address{Department of Mathematics,IIT Patna, Bihta campus, Bihta-801 106}
\curraddr{}
\email{ravindra.pma15@iitp.ac.in}
\thanks{}

\author{Om Prakash}
\address{Department of Mathematics \\
IIT Patna, Bihta campus, Bihta-801 106}
\curraddr{}
\email{om@iitp.ac.in}
\thanks{}



\date{}

\maketitle
\begin{abstract}

 Let $\overline{\Gamma(R)}$ be the complement of zero divisor graph of a finite commutative ring $R$. In this article, we have provided the answer of the question (ii) raised by Osba and Alkam in their paper and prove that $\overline{\Gamma(R)}$ is a divisor graph if $R$ is a local ring. It is shown that when $R$ is a product of two local rings, then $\overline{\Gamma(R)}$ is a divisor graph if one of them is an integral domain. Also, we prove that if $\lvert Ass(R)\rvert = 2$, then $\overline{\Gamma(R)}$ is a divisor graph.
\end{abstract}

\section{INTRODUCTION}

Throughout the manuscript, $R$ represents a finite commutative principal ideal ring with unity. Let $Z(R)$ be the set of zero divisors of $R$ and $Reg(R) = R\backslash Z(R)$, be the set of regular elements of $R$. An element $a \in R$ is called a regular element if there exists $b \in R$ such that $a = a^{2}b$, the element $b$ is called a von Neumann inverse for $a$. In 1980, zero divisor graph was introduced by I. Beck\cite{beck} in context of coloring and redefined by Anderson and Livingston\cite{ander} in 1999. In $R$, a zero divisor graph is denoted by $\Gamma(R)$, and defined by a graph whose vertex set is the set of all non-zero zero divisors of $R$ and any two elements $a, b$ are adjacent if and only if their product is zero i.e $a.b = 0$. Complement graph $\overline{\Gamma(R)}$ is defined on the same vertex set of $\Gamma(R)$ such that two distinct vertices $a$ and $b$ are adjacent if and only if $a.b \neq 0$. Anderson and Livingston also proved that $\Gamma(R)$ is connected with at most diameter $3$.\\

Divisor graph were introduced by Chartrand et al.\cite{div} in 2001. Let $S$ be a nonempty set of positive integers, the divisor graph $G(S)$ of $S$ has vertex set $S$ and two vertices $i$ and $j$ of $G(S)$ are adjacent if $i$ divides $j$ or $j$ divides $i$. A graph $G$ is a divisor graph if there exists a set $S$ of positive integers such that $G$ is isomorphic to $G(S)$. In a directed graph $G$, a vertex is said to be a receiver if its out-degree is zero and its in-degree is positive. A transmitter is a vertex having positive out-degree and zero in-degree. A vertex $t$ with positive in-degree and positive out-degree is transitive if whenever $u \rightarrow t$ and $t \rightarrow v$ are edges in $G$, then $u \rightarrow v$ is an edge in $G$. It has shown in \cite{div} that if $G$ be a graph, then $G$ is a divisor graph if and only if there exists an orientation $D$ of $G$ such that every vertex of $D$ is a transmitter, a receiver, or a transitive vertex. Further, divisor graph are also studied in\cite{al}\\

In this article, we will study the complement graph $\overline{\Gamma(R)}$. We have investigated for what diameter of the ring $R = R_1 \times R_2$ is a divisor graph, where $R_1$ and $R_2$ are commutative rings with unity. Moreover, we observe the condition in which $\overline{\Gamma(R[x]))}$ to be a divisor graph.

The distance between two vertices of a graph $G$ is the number of edges in a minimal path between the vertices. The diameter of $G$ (denoted diam$(G)$) is the maximal distance between any pair of vertices. A cycle is a closed path, excluding loops, from a vertex to itself that does not repeat edges. For basic definitions and results of graph theory, we prefer \cite{harary}.\\

A ring $R$ is said to be a local if it has unique maximal ideal. Let $M$ be an $R$-module, and $P$ a prime ideal of $R$. Then $P$ is an associated prime of $M$ (or that $P$ is associated to $M$) if $P$ is the annihilator of some non-zero $ x \in M$. The set of associated primes of $M$ is denoted by $Ass(M)$.

\section{Main Result}

It is known from Theorem 8.7 of \cite{ati} that if $R$ is a non-local ring, then $R$ is  a direct product of local rings. Therefore, we start our results with local rings and finally present some  results for the direct product of local rings.

\begin{theorem}
	If $R$ is a local ring, then $\overline{\Gamma(R)}$ is a divisor graph.\\
	
\end{theorem}

\begin{proof}
	
	Let $M$ be the maximal ideal of the given local ring $R$. Since $M$ is principal so there exists $a\in R$ such that $M = <a>$ i.e $M = aR$. Since $R$ is finite, there exists $n \in \mathbb{N}$ such that $ ~ a^n =0$ and $a^{n-1} \neq 0$. Now, $V(\overline{\Gamma(R)}) = ua^i$ where $u$ is a unit in $R$ and $i = \{1,...,n-1\}$. Also, two vertices $ua^i$ and $va^j$ are adjacent if and only if $i+j < n$. Let $U = \{u_1, u_2,..., u_m\}$ be the set of units. Also, here $M$ can be written as $\{u_ia^j~\lvert ~i\leq m,~ j< n\}$ (with distinct elements). We consider an orientation for $\overline{\Gamma(R)}$ for any adjacent vertices $u_{i_1}a^j$ and $u_{i_2}a^k$ as follows:\\
	If $j <k$, then the orientation is $u_{i_1}a^j \longrightarrow u_{i_2}a^k$, if $ j= k$ and $i_1< i_2$, then the orientation is $u_{i_1}a^j \longrightarrow u_{i_2}a^k$. It can be easily checked that the under above orientation for $\overline{\Gamma(R)}$, every vertex is either transitive or receiver or transmitter except the isolated vertex of the form $u_i a^{n-1}$. Here, $u_i a^{j}$ is a receiver when $j > \lfloor \frac{n}{2} \rfloor , ~ j \neq n-1$. If  $j < \lceil \frac{n}{2} \rceil , ~ j \neq 1$, then $u_i a^{j}$ is a transitive except $u_m a^{\frac{n-1}{2}}$ (this is a receiver). Also, $u_1 a^1$ is a transmitter. \\
	If $n$ is even, then apart from above transitive/ receiver/transmitter, all vertices of $u_i a^{\frac{n}{2}}$ form are receiver.
\end{proof}

\begin{lemma}
	
	A graph contains the below given induced subgraph $fig(1)$ is not a divisor graph.
	
	\[\begin{tikzpicture}

	\vertex (D) at (4.5,-1) [label=below:${D}$]{};
	\vertex (C) at (6.3,1.5) [label=right:${C}$]{};
	\vertex (B) at (5.4,4) [label=right:${B}$]{};
	\vertex (A) at (3,5) [label=above:${A}$]{};
	\vertex (G) at (0.1,3.5) [label=above:${G}$]{};
	\vertex (F) at (-0.7,1) [label=left:${F}$]{};
	\vertex (E) at (0.8,-1) [label=below:${E}$]{};
	\path
	(A) edge (C)
	(A) edge (D)
	(A) edge (E)
	(A) edge (F)
	(A) edge (G)
	(B) edge (D)
	(B) edge (E)
	(C) edge (E)
	(C) edge (F)
	(C) edge (G)
	(D) edge (E)
	(D) edge (G)
	(E) edge (G)
	;
	
	\end{tikzpicture}\]
	\hspace{5cm} \textbf{Figure 1} \\

\end{lemma}

\begin{proof}
	Let $D$ be an orientation of above graph $fig(1)$ in which every vertex is either receiver or a transmitter or transitive. We assume that $A \longrightarrow G$ in $D$, then $A \longrightarrow F$ in $D$, otherwise orientation does not make sense. Now, we consider the orientation given in $fig(2)$. To complete this orientation we must have either $G \longrightarrow D ~ or ~ D \longrightarrow G$. If $G \longrightarrow D$, then $G$ is neither receiver nor transmitter nor transitive. Again, if $D \longrightarrow G$, then in this case $D$ is neither receiver nor transmitter nor transitive. So, the graph of $fig(1)$ is not a divisor graph.
\end{proof}

\[\begin{tikzpicture}

\tikzstyle{vertex}=[circle, draw, inner sep=0pt, minimum size=6pt]
\tikzset{edge/.style = {->,> = latex'}}

\vertex (D) at (4.5,-1) [label=below:${D}$]{};
\vertex (C) at (6.3,1.5) [label=right:${C}$]{};
\vertex (B) at (5.4,4) [label=right:${B}$]{};
\vertex (A) at (3,5) [label=above:${A}$]{};
\vertex (G) at (0.1,3.5) [label=above:${G}$]{};
\vertex (F) at (-0.7,1) [label=left:${F}$]{};
\vertex (E) at (0.8,-1) [label=below:${E}$]{};

\draw[edge] (A) to (C);
\draw[edge] (A) to (D);
\draw[edge] (A) to (E);
\draw[edge] (A) to (F);
\draw[edge] (A) to (G);
\draw[edge] (B) to (D);
\draw[edge] (B) to (E);
\draw[edge] (C) to (E);
\draw[edge] (C) to (F);
\draw[edge] (C) to (G);
\draw[edge] (D) to (E);
\draw[edge] (E) to (G);
;

\end{tikzpicture}\]
\hspace{5 cm} \textbf{Figure 2} \\

\begin{theorem}
	Let $R_{1}$ and $R_{2}$ be two commutative rings and $R = R_1 \times R_2$ where $diam(\Gamma(R_1)) = diam(\Gamma(R_2)) = 0$. Then $\overline{\Gamma(R)}$ is a divisor graph if and only if at least one of $R_1 $ and $R_2$ is an integral domain.
	
\end{theorem}

\begin{proof}
	
	Let $R_1$ and $R_2$ be integral domains. Then $\Gamma(R)$ is a complete bipartite graph and in this case $\overline{\Gamma(R)}$ has two components in which each of them is complete. So $\overline{\Gamma(R)}$ is a divisor graph. Suppose $R_1$ is an integral domain, then $Z(R_2) = \{0, a\}$, $Reg(R_2) = \{v_1,...,v_n\}$ and $Reg(R_1) = \{u_1,...,u_m\}$, being $R_1$ and $R_2$ are finite. To prove $\overline{\Gamma(R)}$ is a divisor graph, we consider the orientation in $Fig(3)$ as follows:\\
	
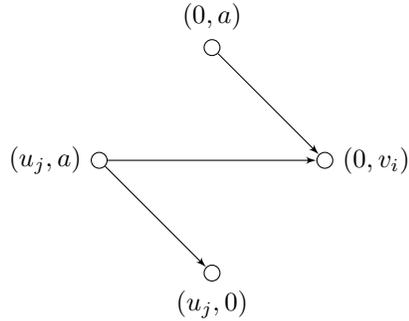
	\[\begin{tikzpicture}
	\tikzstyle{vertex}=[circle, draw, inner sep=0pt, minimum size=6pt]
	%\draw[->,line width=4pt] (0,0) to (1,0);
	\tikzset{edge/.style = {->,> = latex'}}
	
	\vertex (A) at (0,2.5) [label=above:${(0,a)}$]{};
	\vertex (B) at (1.5, 1) [label=right:${(0, v_i)}$]{};
	\vertex (C) at (0, -0.5) [label=below:${(u_j, 0)}$]{};
	\vertex (D) at (-1.5, 1) [label=left:${(u_j, a)}$]{};
	
	%\path
	\draw[edge] (A) to (B);
	\draw[edge] (D) to (B);
	\draw[edge] (D) to (C);
	
	;
	
	\end{tikzpicture}\]
	\hspace{5 cm} \textbf{Figure 3} \\
	
	Define a labelling $f : V(\overline{\Gamma(R))} \longrightarrow \mathbb{N}$ by
	\[ f(x, y) =  \left\{
	\begin{array}{ll}
	2 & (x, y) = (0, a) \\
	2 \times p^i & (x, y) = (0, v_i) \\
	p^j & (x, y) = (u_j, a) \\
	q \times p^j & (x, y) = (u_j, 0) \\
	\end{array},
	\right. \]
	
	where $p$ and $ q$ are distinct primes. Clearly, $f$ is a $1-1$ function and $(x,y)(\alpha , \beta) \neq (0,0)$ if and only if $f(x, y) / f(\alpha , \beta)$ or $f(\alpha, \beta) / f(x , y)$. Hence, $\overline{\Gamma(R)}$ is a divisor graph.\\
	Conversely, let $\overline{\Gamma(R)}$ be a divisor graph. If possible, assume $R_1$ and $R_2$ are not integral domains. Then $Z(R_1) = \{0, b\}$, $Z(R_2) = \{0, a\}$. So, there are distinct vertices $(0, a), (0, v_1), (b, 0), (b, a)$, $(b, v_2), (u_1, 0), (u_2, a)$ are in $\overline{\Gamma(R)}$ and This will give an induced subgraph given in $fig(1)$. Therefore, the graph $\overline{\Gamma(R)}$ is not a divisor graph, which contradict the assumption.
	
\end{proof}

\begin{theorem}
	
	Let $R = R_1 \times R_2$ such that $diam(\Gamma(R_1)) = 0 ~\&~  diam(\Gamma(R_2)) = 1$. Then $\overline{\Gamma(R)}$ is a divisor graph if and only if  $R_1$ is an integral domain.\\
	
\end{theorem}

\begin{proof}
	
	Let $R_1$ be an integral domain and $diam(\Gamma(R_1)) = 1$. Then $Z(R_2) = \{0, x_1,..., x_k\}$, $Reg(R_1) = \{u_1,..., u_m\}$ and $Reg(R_2) = v_1,..., v_n$. Now, we draw a graph $(figure~4)$ with orientation $D$.
	
	\[\begin{tikzpicture}
	
	\tikzstyle{vertex}=[circle, draw, inner sep=0pt, minimum size=6pt]
	\tikzset{edge/.style = {->,> = latex'}}
	% vertices
	
	\vertex (a) at (1.5, 1) [label=right:${(0, x_i)}$]{};
	\vertex (b) at (0, 2.5) [label=above:${(u_j, x_i)}$]{};
	\vertex (c) at (-1.5, 1) [label=left:${(0, v_k)}$]{};
	\vertex (d) at (0, -0.5) [label=below:${(u_j, 0)}$]{};

	%edges
	\draw[edge] (c) to (a);
	\draw[edge] (d) to (b);
	\draw[edge] (c) to (b);

	;
	
	\end{tikzpicture}\]
	\hspace{5 cm} \textbf{Figure 4} \\
	
	Now, we define a function $f : V(\overline{\Gamma(R))} \longrightarrow \mathbb{N}$ by
	
	\[ f(x, y) =  \left\{
	\begin{array}{ll}
	p_i & (x, y) = (0, x_i) \\
	3^{N_{j, i}} & (x, y) = (u_j, x_i) \\
	3^{N_{n, l}} \times \prod\limits_{i=1}^{k} p_i	 & (x, y) = (0, v_k) \\
	3^{N_{n, l}} \times \prod S_{j, i} & (x, y) = (u_j, 0) \\
	\end{array},
	\right. \]\\
	
	where $p_1, p_2,..., p_k, S_{1,1},..., S_{j,i}$ are distinct primes and $N_{1, 1},..., N_{n, l}$ are ascending positive integers. Clearly $f$ is $1-1$ and  $(x,y)(a ,b) \neq (0,0)$ if and only if $f(x, y) / f(a , b)$ or $f(a, b) / f(x , y)$. Hence $\overline{\Gamma(R)}$ is a divisor graph.\\
	Conversely, let $\overline{\Gamma(R)}$ be a divisor graph where $R_1$ is not an integral domain. Then $Z(R_1) = \{0, a\}$, $Reg(R_1) = \{u_1,..., u_m\}$, $Z(R_2) = \{0, x_1,..., x_k\}$ and $Reg(R_2) = v_1,..., v_n$.  Again, we have an induced subgraph of $\overline{\Gamma(R)}$ as in $fig(1)$ and by Lemma(2.1), $\overline{\Gamma(R)}$ is not a divisor graph. Which is a contradiction. Hence $R_1$ is an integral domain.	
	
\end{proof}

\begin{theorem}
	Let $R = R_1 \times R_2$ such that $diam(\Gamma(R_1)) = diam(\Gamma(R_2)) = 1$. Then $\overline{\Gamma(R)}$ is not a divisor graph.
	
\end{theorem}

\begin{proof}
	Since $R_{1}$ and $R_{2}$ are finite rings, let $Z(R_1) = \{0, x_1,..., x_n\}$, $Reg(R_1) = \{y_1,..., y_m\}$, $Z(R_2) = \{0, z_1,..., z_k\}$ and $Reg(R_2) = \{w_1,..., w_l\}$. Now, we find distinct vertices $(0, z_1), (0, w_2)$, $(x_1, 0), (x_2, z_2), (x_3, w_1)$, $(y_1, 0)$ and $(y_2, z_3)$ such that $fig(5)$ is an induced subgraph of $\overline{\Gamma(R)}$. Therefore, by Lemma(2.1), $\overline{\Gamma(R)}$ is not a divisor graph.
\end{proof}

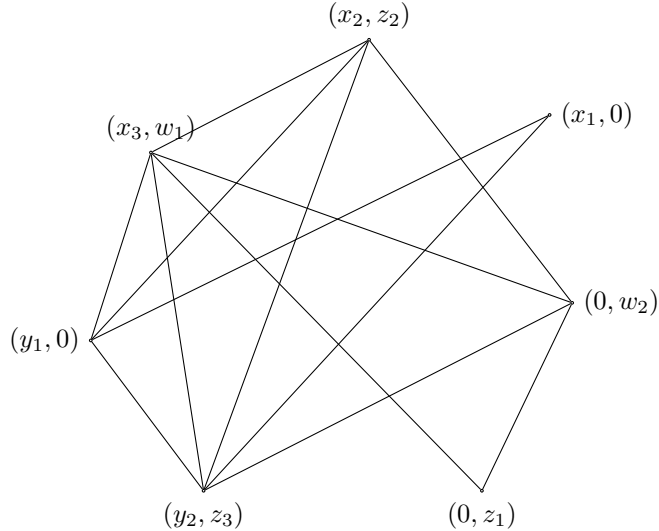
\[\begin{tikzpicture}

\vertex (D) at (4.5,-1) [label=below:${(0, z_1)}$]{};
\vertex (C) at (5.7,1.5) [label=right:${(0, w_2)}$]{};
\vertex (B) at (5.4,4) [label=right:${(x_1, 0)}$]{};
\vertex (A) at (3,5) [label=above:${(x_2, z_2)}$]{};
\vertex (G) at (0.1,3.5) [label=above:${(x_3, w_1)}$]{};
\vertex (F) at (-0.7,1) [label=left:${(y_1, 0)}$]{};
\vertex (E) at (0.8,-1) [label=below:${(y_2, z_3)}$]{};
\path
(D) edge (C)
(G) edge (D)
(A) edge (C)
(C) edge (G)
(C) edge (E)
(B) edge (F)
(B) edge (E)
(A) edge (E)
(A) edge (F)
(A) edge (G)
(G) edge (E)
(F) edge (G)
(E) edge (F)
;

\end{tikzpicture}\]
\hspace{5 cm} \textbf{Figure 5} \\

\begin{theorem}
	
	Let $R_{1}$ and $R_{2}$ be local rings and $R = R_1 \times R_2$ such that $diam(\Gamma(R_1)) = 0$ and  $diam(\Gamma(R_2)) = 2$. Then $\overline{\Gamma(R)}$ is a divisor graph if and only if $R_1$ is an integral domain.\\
	
\end{theorem}

\begin{proof}
	Since $R_{1}$ and $R_{2}$ are finite principal ideal rings, let $Z(R_2)$ be generated by $y \in R_2$ with $y^l = 0, y^{l-1} \neq 0$. Then $Z(R_2) = yR_2$. Let $R_1$ be an integral domain and $Reg(R_1) = \{u_1,..., u_m\}$, $Z(R_2) = \{0, v_1y, v_2y,..., v_ny^{l-1}\}$ and $Reg(R_2) = \{v_1,..., v_n\}$. Clearly, if $(v_iy^j) (v_sy^r) \neq 0$, then $j+r <l$. Consider the following orientation $D$: if $j<r$, then we set $(0, v_iy^j) \longrightarrow (0, v_sy^r)$ and $(u_{\alpha}, v_sy^r) \longrightarrow (u_{\alpha}, v_iy^j)$. If $j = r$, then we set $(0, v_sy^r) \longrightarrow (0, v_iy^j)$ and $(u_{\alpha}, v_iy^j) \longrightarrow (u_{\alpha}, v_sy^r)$ if $s < i$. Here, the vertices $(0, v_i)$ and $(u_j, 0)$ for $1\leq i \leq n-1$, $1\leq j \leq m-1$ are all transitive. Here, $(0,v_n)$ and $(u_m, 0)$ are receivers. If $(0, v_1y^{r_1}) \longrightarrow (0, v_2y^{r_2}) \longrightarrow (\alpha, v_3y^{r_3})$ where $\alpha \in \{0, u_1,..., u_n\}$, then $r_1 +r_2 <l, ~r_2 + r_3 < l$ and so $r_1 +r_3 <l$, since $r_1 < r_2$. By this orientation, $(0, v_1y^{r_1}) \longrightarrow (\alpha, v_3y^{r_3})$. If $(\alpha, v_1y^{r_1}) \longrightarrow (0, v_2y^{r_2}) \longrightarrow (0, v_i)$, then $(\alpha, v_1y^{r_1}) \longrightarrow (0, v_i)$, therefore $(0, v_2y^{r_2})$ is transitive. If $(\alpha, v_1y^{r_1}) \longrightarrow (u_i, v_2y^{r_2}) \longrightarrow (u_j, v_3y^{r_3})$, then $r_1 + r_2 < l, ~ r_2+ r_3 <l$ and so $r_1 + r_3 < l$ because $r_2 > r_3$. Hence, $(\alpha, v_1y^{r_1}) \longrightarrow (u_j, v_3y^{r_3}), ~ \alpha \in \{0, u_1,..., u_n\}$ and finally  from $D$, $(u_i, v_1y^{r_1}) \longrightarrow (u_j, v_2y^{r_2}) \longrightarrow (u_m, 0)$, therefore, $(u_i, v_1y^{r_1}) \longrightarrow (u_m, 0)$. Hence $(u_j, v_2y^{r_2})$ is transitive. Thus, $\overline{\Gamma(R)}$ is a divisor graph.
\end{proof}

\begin{theorem}
	
	Let $R_{1}$ and $R_{2}$ be local rings and $R = R_1 \times R_2$ such that $diam(\Gamma(R_1)) = 1 ~or~ 2$ and  $diam(\Gamma(R_2)) = 2$. Then $\overline{\Gamma(R)}$ is not a divisor graph.\\
	
\end{theorem}

\begin{proof}
	It can easily be proved by Lemma (2.1) that the graph is not a divisor graph.
\end{proof}

\begin{proposition}
	Complement of a complete bipartite graph is a divisor graph.
\end{proposition}

\begin{proof}
	We know by Corollary (2.10) of \cite{div}, complete bipartite graph is a divisor graph. Also, complement of the complete bipartite graph has two components such that each component is a complete graph. Hence, by Proposition $(2.5)$ of \cite{div}, it is a divisor graph.
\end{proof}

\begin{theorem}
	Let $R$ be a commutative ring. If $\lvert Ass(R)\rvert = 2$ and $p_1 \bigcap p_2 = \{0\}$, then $\Gamma(R)$ is a divisor graph.
 \end{theorem}

\begin{proof}
	Let $p_1, p_2 \in Ass(R)$ such that $p_1 \bigcap p_2 = \{0\}$. Since $p_1 \bigcup p_2 \subseteq Z(R)$, let $x \in Z(R)\backslash p_1 \bigcup p_2$. Then there exists $0 \neq y \in R$ such that $xy = 0 \in p_1 \bigcap p_2$. So $y \in p_1 \bigcap p_2 = \{0\}$, which is a contradiction. So $Z(R) = p_1 \bigcup p_2 $.\\
	Now, we have two sets $V_1$ and $V_2$ such that $V_1 = p_1 \backslash \{0\}, ~ V_2 = p_2 \backslash \{0\}$. In order to prove $\Gamma(R)$ is a divisor graph, we prove no two elements of $V_1$ are adjacent and same for $V_2$. If $0 \neq a, b \in V_1$ such that $ab = 0$, then $ab \in p_2$. Since $p_2$ is a prime ideal so either $a \in V_2$ or $b \in V_2$, which is a contradiction. Therefore, $\Gamma(R)$ is a bipartite graph. Now, we take $a \in V_1$ and $b \in V_2$ for all $a, b$. If $ab \in p_1$ since $p_1$ is an ideal and $ab \in p_2$ since $p_2$ is an ideal. Then $ab \in p_1 \bigcap p_2 = \{0\}$ so $ab =0$ and hence $\Gamma(R)$ is a complete bipartite graph. Thus, by $Thm(1.1)$, $\Gamma(R)$ is a divisor graph.
\end{proof}

\begin{theorem}
	Let $R$ be a commutative ring. If $\lvert Ass(R)\rvert = 2$ and $p_1 \bigcap p_2 = \{0\}$, then $\overline{\Gamma(R)}$ is a divisor graph.
\end{theorem}

Few results for the polynomial ring over the ring $R$, we found interesting, specially, in \cite{mc},  if $f(x) \in Z(R[x])$, then there exists a constant $c \in R \backslash \{0\}$ such that $cf(x) = 0$. Also in \cite{osba}, if $R$ is a finite commutative principal ideal ring with unity and $T = R[x]$ or $R[[x]]$, then for each $f \in T$, there exists $c_f \in R, ~ f_1 \in T\backslash Z(T)$ such that $f = c_ff_1$.

\begin{corollary}\cite{osba}
	Let $R$ be a finite commutative principal ideal ring with unity and $T = R[x]$ or $R[[x]]$. Then for each $f, g \in Z(T), ~ fg =0$ if and only if $c_fc_g = 0$.\\
\end{corollary}

\begin{theorem}
	Let $R$ be a finite commutative principal ideal ring with unity and $T = R[x]$. Then $\overline{\Gamma(T)}$ is a divisor graph if $\overline{\Gamma(R)}$ is a divisor graph.
\end{theorem}

\begin{proof}
	Let $\overline{\Gamma(R)}$ be a divisor graph with orientation $D$. If $f(x), g(x) \in Z(T)$ with $f(x)g(x) \neq 0$, then there is an edge between them in $\overline{\Gamma(T)}$. Now from corollary $[2.1]$, $f(x)g(x) = 0$ if and only if $c_fc_g = 0$ so $f(x)g(x) \neq 0$ this implies $c_fc_g \neq 0$. We define an orientation $K$ in $\overline{\Gamma(T)}$. If $c_f \longrightarrow c_g$ in $D$, then $f(x) \longrightarrow g(x)$ in $K$. Let $f(x) = c_ff_1$ be a vertex in $\overline{\Gamma(T)}$. If $c_f$ is a receiver(transmitter) in $D$, then $f(x)$ is clearly a receiver(transmitter) in $K$. Assume $c_f$ is transitive in $D$ with $h(x) \longrightarrow f(x)$ and $f(x) \longrightarrow g(x)$ in $K$. Then $h(x)f(x) \neq 0$ and $f(x)g(x) \neq 0$ and these imply  $c_hc_f \neq 0$ and $c_fc_g \neq 0$. Therefore, $c_h \longrightarrow c_f$, $c_f \longrightarrow c_g$ in $D$. Since $c_f$ is transitive in $D$, $c_h \longrightarrow c_g$ in $D$, therefore, $c_hc_g \neq 0$. Hence $h(x) \longrightarrow g(x)$ in $K$ and $h(x)g(x) \neq 0$, i.e. $f(x)$ is transitive in $K$. Thus, $\overline{\Gamma(T)}$ is a divisor graph.
\end{proof}

\section{ Conclusion }

In this manuscript, we prove that if $R$ is a finite commutative principal ideal ring with unity, then $\overline{\Gamma(T)}$ is a divisor graph if and only if $R$ is a local ring and $R$ is a product of two local rings such that one of them is an integral domain.

\section*{Acknowledgement}
The authors are thankful to Council of Scientific and Industrial Research (CSIR), Govt. of India for financial support under Ref. No. 22/06/2014(i)EU-V, Sr. No. 1061440753 dated 29/12/2014 and Indian Institute of Technology Patna for providing the research facilities.

\bibliographystyle{amsplain}

\end{document}